\documentclass[11pt]{amsart}
\usepackage{amssymb,latexsym,amsmath,amsfonts,amsthm}
\usepackage{amssymb, amscd}
\usepackage{mathrsfs}

\usepackage{marginnote}

\usepackage{color}
\usepackage{enumerate}
\usepackage{soul}
\usepackage{xcolor}

\usepackage{mathtools}

\numberwithin{equation}{section}

\usepackage{graphicx}
\usepackage{amscd}

\usepackage{times}

\usepackage{fullpage}
\DeclareSymbolFont{SY}{U}{psy}{m}{n}
\DeclareMathSymbol{\emptyset}{\mathord}{SY}{'306}

\theoremstyle{plain}

\newtheorem{thm}{Theorem}[section]

\newtheorem{lem}[thm]{Lemma}
\newtheorem{prop}[thm]{Proposition}

\theoremstyle{definition}

\newtheorem{qn}[thm]{Question}
\allowdisplaybreaks

\newcounter{defcounter}
\title[Moment Problem]{Hausdorff moment sequences induced by rational functions}
\author{Md. Ramiz Reza}
\author{Genkai Zhang}
\address[Md. Ramiz Reza and G. Zhang]{Department of Mathematical Sciences, Chalmers University of Technology, G\"{o}teborg}
\email[Md. Ramiz Reza]{ramiz.md@gmail.com}
\email[G. Zhang]{genkai@chalmers.se}

\keywords{Moment problem, positive definite kernel, module tensor product, subnormality} 
\subjclass[2010]{44A60, 46E22, 47B20}

\begin{document}

\thanks{The work of Md. Ramiz Reza was supported by SERB Overseas Post Doctoral Fellowship, SB/OS/PDF-216/2016-2017.}


\begin{abstract}
We study the Hausdorff moment problem for a class of sequences, namely
$(r(n))_{n\in\mathbb Z_+},$ where $r$ is a rational function in the
complex plane. We obtain a necessary condition for such sequence to be
a Hausdorff moment sequence.  We found an interesting connection
between Hausdorff moment problem for this class of sequences with
finite divided differences and convolution of complex exponential
functions. We provide a sufficient condition on the zeros and poles of
a rational function $r$ so that $(r(n))_{n\in\mathbb Z_+}$ is a
Hausdorff moment sequence. G. Misra asked whether the module tensor
product of a subnormal module with the Hardy module over the
polynomial ring is again a subnormal module or not. Using our
necessary condition we answer the question of G. Misra in
negative. Finally, we obtain a characterization of all real
polynomials $p$ of degree up to $4$ and a certain class of real polynomials of degree $5$ for which the sequence $(1/p(n))_{n\in\mathbb Z_+}$ is a Hausdorff moment sequence.

\end{abstract}
\maketitle
\section{Introduction}
Let $\mathcal H_K$ be a reproducing kernel Hilbert space consisting of holomorphic functions on the unit disc $\mathbb D$ with reproducing kernel $K.$ Thus $K(z,w)$ is a complex valued function defined on $\mathbb D\times \mathbb D,$ which is holomorphic in the first variable and anti holomorphic in the second variable and is positive definite in the sense that $(\!(K(z_i,z_j))\!)$ is positive definite for every finite subset $\{z_1,z_2,\ldots,z_n\}$ of unit disc $\mathbb D,$  see \cite{Aronszajn1950theory},\cite{Paulsen2016rkhs}. If the operator $\mathcal M_z$ of multiplication by the coordinate function on the Hilbert space $\mathcal H_K,$  is a bounded operator, then $\mathcal H_K$ is a Hilbert module over the polynomial ring $\mathbb C[z]$ with the module action given by 
\begin{align*}
(p,f) : \mathbb C[z] \times \mathcal H_K\mapsto (p(\mathcal M_z))(f)\in \mathcal H_K.
\end{align*}
See \cite{Douglasmodule} for more discussion on Hilbert modules. 
If the operator $\mathcal M_z$  is a contraction, that is, $\|\mathcal M_z\|_{op}\leq 1,$ then $K$ is said to be a contractive kernel. In such case, using Von-Neuman inequality, we have that the Hilbert module $\mathcal H_K$  is also contractive, that is, $\|pf\|_{\mathcal H_K}\leq \|p\|_{\infty}\|f\|_{\mathcal H_K},$ where $\|p\|_{\infty}$ denotes the supremum norm of $p$ on the unit disc $\mathbb D.$ It is well known that a reproducing kernel $K$ is contractive if and only if the function $(1-z\bar{w})K(z,w)$ is positive definite, see \cite[Theorem 5.21]{Paulsen2016rkhs}.

If the operator $\mathcal M_z$  is a bounded subnormal operator, that is, the operator $\mathcal M_z$ has a normal extension to a larger Hilbert space $\mathcal K$ containing $\mathcal H_K,$  then the Hilbert module $\mathcal H_K$ is said to be a subnormal module or equivalently $K$ is said to be a subnormal kernel, see \cite{Conwaysubnormal} for more details on subnormal operators. 

Given a reproducing kernel $K$ it is often difficult to determine whether it is a subnormal kernel or not. But there is a simple characterization of contractive subnormal kernels in terms of Hausdorff moment sequences if the kernel $K$ has the following diagonal form: 
\begin{align}\label{diagonal kernel}
K(z,w)&= \sum\limits_{j=0}^{\infty} a_j (z\bar{w})^j,\,\,\,\,z,w\in\mathbb D,
\end{align}
where $a_j> 0$ for all $j\in\mathbb Z_+.$ In such case  the set $\{\sqrt{a_n}z^n : n\in\mathbb Z_+\}$ forms an orthonormal basis for $\mathcal H_K$ and the operator $\mathcal M_z$ is a unilateral weighted shift associated  with the weight sequence $ \big(\sqrt{\tfrac{a_n}{a_{n+1}}}\big)_{n\in\mathbb Z_+},$ see \cite[Theorem 4.12]{Paulsen2016rkhs}.

A sequence $(x_n)_{n\in\mathbb{Z}_{+}}$ of positive numbers is said to be a Hausdorff moment sequence if there exists a positive Radon measure $\mu$ supported on the interval $[0,1]$ such that  
\begin{align*}
x_n &=\int_{0}^1 t^n d\mu(t),\,\,\,n\in \mathbb{Z}_{+}.
\end{align*}
F. Hausdorff characterized the Hausdorff moment sequences by complete monotonicity \cite{Hausdorff}. 
\begin{thm}[F. Hausdorff]\label{Hausdorff criterion}
A sequence $(x_n)_{n\in\mathbb{Z}_{+}}$ of positive number is a Hausdorff moment sequence if and only if the sequence $(x_n)_{n\in\mathbb{Z}_{+}}$ is completely monotone, that is,
\begin{align*}
\sum\limits_{j=0}^m (-1)^j \binom{m}{j} x_{n+j} \geq 0,\,\,\,\,m\in\mathbb N, n\in \mathbb Z_+.
\end{align*}
\end{thm}
The following theorem characterizes all contractive subnormal kernels of the form \eqref{diagonal kernel} in terms of Hausdorff moment sequences, see \cite[Theorem 6.10]{Conwaysubnormal}.
\begin{thm}
A reproducing kernel $K(z,w)= \sum_{j=0}^{\infty} a_j (z\bar{w})^j,$ defined on unit disc $\mathbb D$ is a contractive subnormal kernel if and only if the sequence $\big(\tfrac{1}{a_n}\big)_{n\in\mathbb Z_+}$ is a Hausdorff moment sequence.
\end{thm}

It is well known that if $K_1$ and $K_2$ are two reproducing kernels then their product $K_1 K_2$ is also a reproducing kernel, see \cite{Aronszajn1950theory}. In 1988, N. Salinas introduced the notion of sharp reproducing kernels and relates the functional Hilbert space $\mathcal H_{K_1K_2}$ to the module tensor product of $\mathcal H_{K_1}\otimes _{\mathbb C[z]}\mathcal H_{K_2}$ over the algebra of polynomials, see \cite{Salinas1988products}. In the same article, N. Salinas asked whether $\mathcal H_{K_1}\otimes _{\mathbb C[z]}\mathcal H_{K_2}$ is again subnormal if $K_1$ and $K_2$ are subnormal kernel. 

Let $S(z,w)=(1-z\bar{w})^{-1}$ be the Hardy kernel on the unit disc $\mathbb D,$ which is a contractive subnormal kernel. It is well known that the module tensor product $\mathcal H_{S}\otimes _{\mathbb C[z]}\mathcal H_{K}$ is contractive if $K$ is a sharp kernel, see \cite[Theorem 5.21]{Paulsen2016rkhs} and \cite{Salinas1988products}. G. Misra asked a special version of the question by N. Salinas. He asked whether the module tensor product $\mathcal H_{S}\otimes _{\mathbb C[z]}\mathcal H_{K}$ is again subnormal or not  if the kernel $K$ is a subnormal kernel. Anand and Chavan answer the question of N. Salinas in negative, see \cite{SameerMTP}. They consider the following class of contractive subnormal kernel functions:
\begin{align*}
K_s(z,w)&= \sum\limits_{j=0}^{\infty} (sj+1)(z\bar{w})^j,\,\,\,z,w\in\mathbb D,s>0.
\end{align*}
They obtain the following characterizations of Hausdorff moment sequences induced by degree $3$ polynomials with real coefficients.
\begin{thm}[Anand-Chavan]\label{AC deg 3}
Let $p$ be a real polynomial of the form $p(z)= (z+1)(z-\alpha)(z-\beta).$ Then the sequence $(1/p(n))_{n\in\mathbb Z_+}$ is a Hausdorff moment sequence if and only if 
\begin{enumerate}
\item $\alpha < 0, \beta < 0,\,\,\,$ when $\alpha, \beta \in\mathbb R,$
\item $\Re(\alpha) < -1,\,\,\,$ when $\alpha$ is a non real complex number and $\beta = \bar{\alpha}.$
\end{enumerate}
\end{thm}
Using this characterization they completely determine when the product kernel $K_{s_1}K_{s_2}$ is again a subnormal kernel and produce a family of counterexamples to answer the question of N. Salinas in negative. But this class fails to answer the question of G. Misra. It turns out that the kernel function $SK_s$ is again a contractive subnormal kernel for every $s>0.$ Note that if the reproducing kernel $K$ is of the form \eqref{diagonal kernel} then the product kernel $SK$ is given by 
\begin{align*}
SK(z,w)&=\sum\limits_{n=0}^{\infty} \bigg(\sum\limits_{j=0}^{n}a_j\bigg)(z\bar{w})^n,\,\,\,z,w\in\mathbb D.
\end{align*} It follows that the kernel function $SK$ is always a contractive kernel, see \cite[Theorem 5.21]{Paulsen2016rkhs}. Thus in a special case the question of G. Misra reduces to the following equivalent question:
\begin{qn}\label{question GM}
If the sequence $(1/a_n)_{n\in\mathbb Z_+}$ is a Hausdorff moment sequence, then does it follow that the sequence $\big(\tfrac{1}{a_0+\cdots+a_{n}}\big)_{n\in\mathbb Z_+}$ is also a Hausdorff moment sequence?
\end{qn}
A similar kind of transformation of Hausdorff moment sequences had been discussed in \cite{BergDuranSTHM} by Berg and Dur\'{a}n. They show that
\begin{thm}[Berg-Duran]
If $(a_n)_{n\in\mathbb Z_+}$ is a Hausdorff moment sequence, then $\big(\tfrac{1}{a_0+\cdots+a_{n}}\big)_{n\in\mathbb Z_+}$ is also a Hausdorff moment sequence.
\end{thm}

In this article we answer the question of G. Misra in negative. Motivated by the classification result, namely the Theorem \ref{AC deg 3}, we concentrate our focus on a special class of sequences, namely, the sequences induced by  rational functions, 
\begin{align}\label{rational sequence}
x_n= \frac{q(n)}{p(n)},\,\,\,\,n\in\mathbb Z_+,
\end{align} where $q,p$ are polynomials with real coefficients and $\deg(q) < \deg (p).$ Our aim is to describe all such polynomials $q,p$ (in terms of their zeros) so that the sequence $(x_n)_{n\in\mathbb{Z}_{+}}$ is a Hausdorff moment sequence. 

In order to study joint subnormality  of the spherical Cauchy dual of a balanced joint $q$-isometry multishift, Anand and Chavan started investigating the class of sequences $(1/p(n))_{n\in\mathbb{Z}_{+}},$ where $p$ is a polynomial satisfying $p(\mathbb R_+)\subseteq \mathbb R_+,$ see \cite{SameerJI}. There they draw upon an interesting connection between Hausdorff moment problem, Hermite interpolation and divided differences of exponential functions. Here we have obtained a natural generalization of those results to the case of rational functions. 
We consider real polynomials $p$ having roots in the left half plane $\mathbb H_-:=\{z\in\mathbb C: \Re(z) < 0\}.$ Assume that the polynomial $p$ is of the following form. 
\begin{align}
p(z)&=\prod\limits_{j=1}^m (z-\alpha_j)^{b_j},
\end{align}
where $\alpha_j$'s are distinct complex number in $\mathbb H_-$ and $b_j\in\mathbb N$ for $j=1,2,\ldots,m.$ Let $q$ be an arbitrary real polynomial with $\deg(q) < \deg(p)$ and $q$ has no common zero with $p.$ First we find that there exists a real valued weight function $w_{q,p}(t),$ which is in $L^1[0,1]$ and continuous on $(0,1]$ so that  
\begin{align*}
\frac{q(n)}{p(n)}&= \int_0^1 t^n w_{q,p}(t)dt,\,\,n\in\mathbb Z_+.
\end{align*}
By uniqueness of the representing measure associated to a Hausdorff moment sequence it follows that the sequence $(q(n)/p(n))_{n\in\mathbb Z_+}$ is a Hausdorff moment sequence if and only if the weight function $w_{q,p}(t) \geq 0$ for all $t\in(0,1].$

There is a natural connection between the weight function $w_{q,p}(t)$ and the general divided differences of the family of function $F_t(z):= q(z)t^{-z-1}$ defined on $\mathbb H_{-}.$ We show that 
\begin{thm}\label{WTDD}
The weight function $w_{q,p}(t)$ is given by
\begin{align*}
w_{q,p}(t)&= F_t[\underbrace{\alpha_1,\ldots,\alpha_1}_{b_1\text{\,times}},\underbrace{\alpha_2,\ldots,\alpha_2}_{b_2\text{\,times}},\ldots,\underbrace{\alpha_m,\ldots,\alpha_m}_{b_m\text{\,times}}],\,\,\,\,t\in(0,1].
\end{align*}
\end{thm}

Minimal Hausdorff moment sequences and  completely monotonic functions are intimately related, see \cite[Chapter 4]{widderlaplace}. In our case, for any rational function $r(x)= q(x)/p(x)$ with poles in $\mathbb H_-$ the associated weight function $w_{q,p}(t)$ is in  $L^1[0, 1]$ and hence it follows that if $\{r(n)\}_{n\in\mathbb Z_+}$ is a Hausdorff moment sequence then necessarily $\{r(n)\}_{n\in\mathbb Z_+}$ is a minimal Hausdorff moment sequence. Thus $\{r(n)\}_{n\in\mathbb Z_+}$ is a Hausdorff moment sequence if and only if $r(x)$ is completely monotonic function on $[0,\infty).$  K. Ball provided a sufficient condition on the zeros and poles of a rational function $r$ of the form $r(x)=\prod_{i=1}^m(x+z_i)/(x+p_i)$ so that $r(x)$ is a completely monotone function on $[0,\infty)$, see \cite{ball94cmrf}. Using Theorem \ref{WTDD}, we find the following sufficient condition on the zeros and poles of a rational function $r$ so that $(r(n))_{n\in\mathbb Z_+}$ is a Hausdorff moment sequence.

\begin{prop}\label{sufficient for rational}
Let $p$ be a real polynomial having distinct negative real roots, say $\alpha_1,\alpha_2,\ldots,\alpha_m$ with no multiplicity and $q$ be any real polynomial such that $\deg(q) < \deg (p)$ and $q, p$ has no common zero. If the divided difference $q[\alpha_1,\alpha_2,\ldots,\alpha_j]\geq 0$ for $j=1,2,\ldots,m,$ then the sequence $\{q(n)/p(n)\}_{n\in\mathbb Z_+}$ is a Hausdorff moment sequence.
\end{prop}

In this article, we obtain the following necessary condition for a sequence $(x_n)_{n\in\mathbb Z_+}$ of the form \eqref{rational sequence} to be a Hausdorff moment sequence.
\begin{thm}\label{not moment criterion}
Let $p$ be a real polynomial having at least one non real root and of the form
\begin{align*}
p(z)&=\prod_{i=1}^s (z-r_i)^{l_i}\prod_{i=1}^d (z-z_i)^{m_i}\prod_{i=1}^d (z-\bar{z}_i)^{m_i},
\end{align*} where $r_i<0,$ for $i=1,2,\ldots,s$ and $z_k=x_k+iy_k$ with $x_k<0, y_k>0$ for $k=1,2,\ldots,d.$ Let $q$ be any other polynomial having no common zeros with $p$ and $\deg(q) < \deg (p).$ 

If there exists a non real root, say $z_1,$ (without loss of generality) of $p$,  such that $\Re(z_1)> r_i$ for every $i=1,2,\ldots,s$ and $\Re(z_1)> \Re(z_i)$ for every $i=2,3,\ldots,d$ then the sequence $(q(n)/p(n))_{n\in\mathbb Z_+}$ is not a Hausdorff moment sequence.
\end{thm}

To answer the question of G. Misra, we consider the following class of kernel functions
\begin{align*}
K_c(z,w)&= \sum_{n=0}^{\infty} (n+c)^6 z^n\bar{w}^n,\,\,\,z,w\in\mathbb D,\,c>0.
\end{align*}
We find that the kernel function $K_c$ is a contractive subnormal kernel for every $c>0$ and then using the necessary condition in Theorem \ref{not moment criterion}, we show that the product kernel $SK_c$ is not a subnormal kernel for all $c$ in a neighborhood of $1.$ This answer the question of G. Misra and as well as the question of N. Salinas in negative.

The following result as shown in \cite[Lemma 2.1]{BergDuranSTHM} is often helpful in determining whether a given sequence is a Hausdorff moment sequence or not.
\begin{lem}\label{Berg Duran product theorem}
Let $(x_n)_{n\in\mathbb{Z}_+}$ and $(y_n)_{n\in\mathbb{Z}_+}$ be two Hausdorff moment sequences associated to some Radon measure $\mu$ and $\nu$ respectively. Then the sequence $(x_ny_n)_{n\in\mathbb{Z}_+}$ is also a Hausdorff moment sequence and associated with the  convolution measure $\mu\diamond \nu.$ 
\end{lem}

It is straightforward to see that the sequence $\{\tfrac{1}{n-a}\}_{n\in\mathbb{Z}_{+}}$ is a Hausdorff moment sequence associated with the measure $t^{-a-1}dt$ on $[0,1]$ for all $a<0.$ Now using Lemma\ref{Berg Duran product theorem}, it follows that, if $p$ is a reducible polynomial over $\mathbb R$ having zeros in $\mathbb H_{-},$ then the sequence $(1/p(n))_{n\in\mathbb Z_+}$ is a Hausdorff moment sequence, see also \cite[Theorem 3.1]{SameerJI}. This leads us to the question of classifying all real polynomial $p$, which are not reducible over $\mathbb R$ so that the sequence $(1/p(n))_{n\in\mathbb Z_+}$ is a Hausdorff moment sequence.

Using Lemma \ref{Berg Duran product theorem}, Theorem \ref{AC deg 3},
Theorem \ref{not moment criterion} and taking consideration a special
case separately, we obtain the complete list of all real polynomial
$p$ up to degree $4$  for which the sequence $(1/p(n))_{n\in\mathbb Z_+}$ is a Hausdorff moment sequence. For degree $5$ polynomial $p,$ we do not have have the complete classification result. Except when $p$ is of the form 
\begin{align*}
p(z)&= (z-r)(z-\alpha)(z-\bar{\alpha})(z-\beta)(z-\bar{\beta}),
\end{align*} where $\Re(\alpha),\Re(\beta) \leq r <0,$ we have answer for all other polynomial $p$ of degree $5.$ In the very special case when all the roots of $p$ lie in a vertical line, that is when $\Re(\alpha)=\Re(\beta)=r <0,$ we have the following interesting result:
\begin{thm}\label{deg 5 special}
Let $p$ be a real polynomial of the form 
\begin{align*}
p(z):= (z-r)\prod \limits _{j=1}^2(z-(r+iy_j))(z-(r-iy_j)),\,\text{where \,}r<0,\,\text{and\,}0<y_1\leq y_2.
\end{align*}
Then the sequence $1/p(n)$ is a Hausdorff moment sequence if and only if the ratio $y_2/y_1$ is a positive integer greater than $1.$ 
\end{thm}

The organization of the paper is as follows. In Section 2, we obtain a partial decomposition formula for rational functions of the form $q/p,$ where $q,p$ are polynomial with $\deg q < \deg p.$ Using this decomposition we show that any sequence $(x_n)_{n\in\mathbb Z_+}$ of the form \eqref{rational sequence},  where all the roots of $p$ lies in $\mathbb H_{-},$ must be a moment sequence associated with a real measure (not necessarily positive) of the form $w_{q,p}(t)dt$ on the interval $[0,1]$ for some continuous function $w_{q,p}(t)$ on $(0,1].$

We established the relationship between the weight function $w_{q,p}(t)$ and the general divided difference of the function $F_t(z):= q(z)t^{-z-1}$ defined on $\mathbb H_{-}$ in Section 3. In the case when $q\equiv 1,$ the weight function $w_{p}(t)$ is intimately related with the convolution of complex exponential function $f_j(y):= \exp(\alpha_jy)$ defined on $\mathbb R_+,$ where $\alpha_j$'s are the roots of $p$ in $\mathbb H_{-}.$ We also explore this relationship in Section 3.

In Section 4, we obtain the necessary condition for a sequence $(x_n)_{n\in\mathbb Z_+}$ of the form \eqref{rational sequence} to be a Hausdorff moment sequence, see Theorem \ref{not moment criterion}. Using this condition we produce a family of counterexample to answer the question of G. Misra and N. Salinas in negative. Theorem \ref{rational necessary} gives us another necessary condition when all the roots of $p$ are negative real number.

Section 5 is devoted to the discussion of characterizing all real
polynomial $p$ up to degree $5$ for which the sequence $(1/p(n))_{n\in\mathbb Z_+}$ is a Hausdorff moment sequence. 



\section{ Partial Fraction Decomposition of Rational Function and Weight function} 
We start by establishing a formula for the partial fraction decomposition of a rational function. For a polynomial $p,$ the partial fraction decomposition of $1/p$ is obtained in \cite[Lemma 3.3]{SameerJI}, using generalized Hermite's interpolation formula, see \cite{Spitzbart}. We obtain a generalized formula for partial fraction decomposition of a rational function.  For a holomorphic function $f$ on the complex plane and for $k\in\mathbb{Z}_+,$ we use the notation $f^{(k)}(w)$ to denote $k$th order complex derivative of the $f$ at the point $w,$ that is, $f^{(k)}(w)=\frac{\partial^k}{\partial z^k}f(z)|_{z=w}.$
\begin{prop}[Partial fraction decomposition formula]\label{PDF Theorem}
Let $p$ be a polynomial of the form 
\begin{align}\label{general polynomial form}
p(z)&=\prod\limits_{j=1}^m (z-\alpha_j)^{b_j},\,\,p_k(z)= \prod\limits_{j=1,j\neq k}^m (z-\alpha_j)^{b_j},\,\,\,k=1,2,\ldots,m,
\end{align}
where $\alpha_j$'s are distinct complex number and $b_j\in\mathbb N$ for $j=1,2,\ldots,m.$   Let $q$ be another polynomial such that $\deg (q) < \deg (p)$ and $q$ has no common zero with $p.$ Then it follows that 
\begin{align*}
\frac{q(z)}{p(z)}&= \sum\limits_{i=1}^m\sum\limits_{j=1}^{b_i} A_i^j\frac{1}{(z-\alpha_i)^j},\,\,A_i^j=\frac{\big(\tfrac{q}{p_i}\big)^{(b_i-j)}(\alpha_i)}{(b_i-j)!},\,\,i=1,2,\ldots,m,\,\,j=1,2,\ldots,b_i. 
\end{align*}
\end{prop}
\begin{proof}
Let $d$ be the degree of the polynomial $p.$ So, $d= \sum_{i=1}^{m} b_i.$ Let $f_i$ be the polynomial given by
\begin{align*}
f_i(z)&= \sum\limits_{j=1}^{b_i} A_i^j\frac{p(z)}{(z-\alpha_i)^j}= p_i(z)\sum\limits_{j=1}^{b_i} A_i^j (z-\alpha_i)^{b_i-j} =p_i(z)h_i(z),
\end{align*} where $h_i(z)=\sum_{j=1}^{b_i} A_i^j (z-\alpha_i)^{b_i-j}$  and $A_i^j=\frac{\big(\tfrac{q}{p_i}\big)^{(b_i-j)}(\alpha_i)}{(b_i-j)!},\,\,i=1,2,\ldots,m,\,\,j=1,2,\ldots,b_i.$   
Let $r$ be the polynomial given by 
\begin{align*}
r(z) &= \sum\limits_{i=1}^mf_i(z)= \sum\limits_{i=1}^m\sum\limits_{j=1}^{b_i} A_i^j\frac{p(z)}{(z-\alpha_i)^j}.
\end{align*}
To prove the Proposition it is sufficient to show that $r(z)\equiv q(z).$ Now we will show 
\begin{align}
f_i^{(k)}(\alpha_i)&= q^{(k)}(\alpha_i),\,\,\,k=0,1,\ldots,b_i-1,\label{sub 1} \\
f_j^{(k)}(\alpha_i)&=0,\,\,\,j\neq i, \,k=0,1,\ldots,b_i-1. \label{sub 2}
\end{align} 
This will give us 
\begin{align*}
r^{(k)}(\alpha_i)&= q^{(k)}(\alpha_i),\,\,\,k=0,1,\ldots,b_i-1,\,\,i=1,2,\ldots,m.
\end{align*}
Note that each $f_i$ is a polynomial of degree at most $d-1.$ So, $\deg(r)$ is at most $d-1.$ By our assumption $\deg(q)$ is at most $d-1.$ Hence by uniqueness of the interpolating polynomial we obtain that $r(z)\equiv q(z).$ 

To prove $\eqref{sub 1}$ first observe that for an arbitrary but fixed $i\in\{1,2,\ldots,m\},$  each $p_j(z)$ has $(z-\alpha_i)^{b_i}$ as a factor for all $j\neq i.$ Hence each $f_j(z)$ has $(z-\alpha_i)^{b_i}$ as a factor for $j\neq i$ and for each $i=1,2,\ldots,m.$ It follows trivially that $f_j^{(k)}(\alpha_i)=0,\,\,\,j\neq i, \,k=0,1,\ldots,b_i-1.$ The claim in $\eqref{sub 2}$ will be proved once the following lemma is established.
\begin{lem}
Let  $h(z)$ and $\tfrac{q(z)}{h(z)}$ be well defined holomorphic function in a neighborhood of $\alpha$, and $B_j=\tfrac{\big(\tfrac{q}{h}\big)^{(n-j)}(\alpha)}{(n-j)!},$ for $j=1,2,\ldots,n.$
Let $L(z)= h(z)\big(\sum_{j=1}^{n} B_j (z-\alpha)^{n-j}\big).$
  Then it follows that
\begin{align*}
L^{(k)}(\alpha)&= q^{(k)}(\alpha),\,\,k=0,1,\ldots,(n-1).
\end{align*}
\end{lem}
\begin{proof}
Let $t(z)=\sum_{j=1}^{n} B_j (z-\alpha)^{n-j}.$ Then $t^{(k)}(\alpha)=k!B_{n-k}$ for $k=1,2,\ldots,(n-1).$ Applying Leibniz's differentiation rule to the function $h(z)t(z)$ and to the function $h(z)\tfrac{q(z)}{h(z)}$ we get that
\begin{align*}
L^{(k)}(\alpha)&= \sum\limits_{j=0}^{k} \binom{k}{j} h^{(k-j)}(\alpha)t^{(j)}(\alpha)= \sum\limits_{j=0}^{k}\binom{k}{j} h^{(k-j)}(\alpha) B_{n-j}(\alpha)j!\\
& = \sum\limits_{j=0}^{k}\binom{k}{j} h^{(k-j)}(\alpha)\big(\frac{q}{h}\big)^{(j)}(\alpha)= \big(h . \tfrac{q}{h}\big)^{(k)}(\alpha)=q^{(k)}(\alpha),\,\,\,k=0,1,\ldots,(n-1).
\end{align*} 
\end{proof}
Applying this Lemma to $f_i(z)=p_i(z)h_i(z),$ for $i=1,2,\ldots,m,$ we get our desired claim in $\eqref{sub 2}.$ Hence the proof of the Proposition $\ref{PDF Theorem}$ follows.
\end{proof}

Let $p$ and $q$ be two polynomial with real coefficients such that $\deg(q) < \deg(p).$ Let $p$ be of the form \eqref{general polynomial form}.
Let us further assume that $p$ is a stable polynomial, that is, $\alpha_i\in \mathbb H_-$ in the left half plane,  for all $i=1,2,\ldots,m.$ We also assume that $q$ and $p$ have no common root and $q(n)>0$ for all $n\in\mathbb{Z}_{+}.$  We consider the sequence $x_n =(q(n)/p(n)),\,n\in\mathbb{Z}_{+}.$ Now we would determine when the sequence $(x_n)_{n\in\mathbb{Z}_{+}}$ is a Hausdorff moment sequence. For $x\in\mathbb{R}_+$ and $\alpha\in\mathbb C$ with $\Re(\alpha)<0,$ it is clear that
\begin{align*}
\frac{1}{x-\alpha} &= \int_{0}^1 t^{x-\alpha-1}dt.
\end{align*}
Differentiating both side repeatedly with respect to $x,$  we get 
\begin{align*}
\frac{1}{(x-\alpha)^k} &= \int_{0}^1\frac{t^{x-\alpha-1}(\log(1/t))^{k-1}}{(k-1)!}dt,\,\,\,k\in\mathbb{N},\,x\in\mathbb R_+.
\end{align*}
Now using partial decomposition formula in Proposition $\ref{PDF Theorem},$ we obtain that 
\begin{align*}
\frac{q(n)}{p(n)} &= \int_0^1 t^n w_{q,p}(t)dt,
\end{align*} 
where the weight function $w_{q,p}(t)$ is given by 
\begin{align}\label{weight function another form}
w_{q,p}(t)= \sum\limits_{i=1}^m\bigg(\sum\limits_{j=1}^{b_i}\frac{A_i^j(\log(1/t))^{j-1}}{(j-1)!}\bigg)t^{-\alpha_i-1}, \,\,\,t\in (0,1].
\end{align}
Since $p$ is a real polynomial it follows that if $\alpha$ is a complex root of $p$ with multiplicity $k,$ then its conjugate $\bar{\alpha}$ is also a root of $p$ with same multiplicity $k.$ Let $r_1,r_2,\ldots,r_s$ be the real roots of $p$ with multiplicity $l_1,l_2,\ldots,l_s$ respectively. Let $z_1,z_2,\ldots,z_d$ along with their conjugates be the complex roots of $p$ with multiplicity $m_1,m_2,\ldots,m_d$ respectively. So, we have $m= s+ 2d.$ Let us enumerate $\alpha_i$'s in \eqref{general polynomial form} in the following way. $\alpha_i=r_i$ for $i=1,2,\dots,s$ and $\alpha_{s+i}= z_i, \alpha_{s+d+i}= \bar{z}_i$ for $i=1,2,\ldots,d.$ Consequently, we have that $b_i=l_i$ for $i=1,2,\ldots,s$ and $b_{s+i}=m_i=b_{s+d+i}$ for $i=1,2,\ldots,d.$ Thus the polynomial $p$ is of the following form
\begin{align}\label{full general form of p}
p(z)&=\prod_{i=1}^s (z-r_i)^{l_i}\prod_{i=1}^d (z-z_i)^{m_i}\prod_{i=1}^d (z-\bar{z}_i)^{m_i},
\end{align} where $r_i<0,$ for $i=1,2,\ldots,s$ and $z_k=x_k+iy_k$ with $x_k<0, y_k>0$ for $k=1,2,\ldots,d.$ For such $p$ the weight function $w_{q,p}(t)$ can be rewritten in the following form,
\begin{align*}
w_{q,p}(t)= \sum\limits_{i=1}^s\bigg(\sum\limits_{j=1}^{l_i}\frac{A_i^j(\log(1/t))^{j-1}}{(j-1)!}\bigg)t^{-r_i-1} + 2 \Re \bigg(\sum\limits_{i=1}^d\bigg(\sum\limits_{j=1}^{m_i}\frac{A_{s+i}^j(\log(1/t))^{j-1}}{(j-1)!}\bigg)t^{-z_i-1}\bigg). 
\end{align*}
Let $z_k= x_k+iy_k$ for $k=1,2,\ldots,d.$ Let $\theta_i^j $ be the principal argument of $A_{s+i}^j$ for $i=1,2,\ldots,d$ and $j=1,2,\ldots,m_i.$ Thus $\Re (A_{s+i}^j t^{-z_i-1})= |A_{s+i}^j|t^{-x_i-1}\cos (\theta_i^j-y_i\log t).$ Consequently, the weight function $w_{q,p}(t)$ takes the following form. 

\begin{multline}\label{weight function}
w_{q,p}(t)= \sum\limits_{i=1}^s\big(\sum\limits_{j=1}^{l_i}\frac{A_i^j(\log(1/t))^{j-1}}{(j-1)!}\big)t^{-r_i-1} +\\
 \sum\limits_{i=1}^d\big(\sum\limits_{j=1}^{m_i}\frac{|A_{s+i}^j|(\log(1/t))^{j-1}2\cos(\theta_i^j-y_i\log t)}{(j-1)!}  \big)t^{-x_i-1}. 
\end{multline}
Note that $w_{q,p}(t)$ is a continuous function on $(0,1]$ and $w_{q,p}(t)$ is also integrable on $[0,1].$ Thus the sequence $(q(n)/p(n))_{n\in\mathbb Z_+}$ is a Hausdorff moment sequence if and only if $w_{q,p}(t)dt$ is a positive measure on $[0,1]$, i.e.,  if and only if the weight function $w_{q,p}(t)\geq 0$ for all $t\in[0,1]$.

It has been shown in \cite{SameerMTP} that for a sequence $(1/p(n))_{n\in\mathbb Z_+},$ where $p$ is a degree $3$ polynomial with real coefficients, to be a Hausdorff moment sequence it is necessary that all the roots of $p$ must lie in $\mathbb H_{-}$ (see also \cite{SameerJI}). Since we already have the weight function $w_{q,p}(t)$ explicitly in \eqref{weight function}, the same arguments as in  \cite{SameerMTP} can also be applied in general to deduce that for a sequence $(x_n)_{n\in\mathbb Z_+}$ of the form \eqref{rational sequence} to be a Hausdorff moment sequence it is  necessary that all the roots of $p$ must lie in $\mathbb H_{-}.$ So, from now on throughout the article we will always assume that roots of $p$ lies in $\mathbb H_{-}.$ 

\section{Weight function, finite divided differences and Convolution}
Now we will provide an alternative expression of the weight function  $w_{q,p}(t)$ in terms of finite divided differences of the function $F_t(z):= q(z)t^{-z-1}$ defined on $\mathbb H_{-}.$ In \cite[Proposition 3.7]{SameerJI}, this already has been established in the case when all the roots of $p$ lies in $\mathbb R_-$ and $q\equiv 1.$ We obtain a natural generalization of their result. Let $F$ be a complex valued function on complex plane and $x=(x_1,x_2,\ldots,x_m)\in \mathbb C^m$
 be $m$ distinct complex number.
Let $P$ and $P_j$ be the function defined by 
\begin{align*}
P(z,x)&=\prod_{k=1}^{m}(z-x_k),\,\,\,P_j(z,x)=\prod_{k=1,k\neq j}^{m}(z-x_k),\,\,j=1,2,\ldots,m.
\end{align*}
The divided difference $F[x_1,x_2,\ldots,x_m]$ on $m$ distinct points $x_1,x_2,\ldots,x_m$ of an one variable function $F$  is defined by
\begin{align*}
F[x_1,x_2,\ldots,x_m]:= \sum\limits_{j=1}^m \frac{F(x_j)}{P_j(x_j,x)}.
\end{align*}
The general divided difference with repeated arguments of the function $F$ is defined by
\begin{align*}
F[\underbrace{x_1,\ldots,x_1}_{(r_1+1)\text{times}},\underbrace{x_2,\ldots,x_2}_{(r_2+1)\text{times}},\ldots,\underbrace{x_m,\ldots,x_m}_{(r_m+1)\text{times}}]:=\frac{1}{r_1!r_2!\ldots,r_m!}\frac{\partial^{r_1+r_2+\cdots+r_m}}{\partial x_1^{r_1}\partial x_2^{r_2}\ldots\partial x_m^{r_m}}F[x_1,x_2,\ldots,x_m].
\end{align*}
Now we compute the general divided differences for the family of function $\{F_t(z):t\in(0.1]\}$ given by $F_t(z)=q(z)t^{-z-1}$ for  $\Re(z)<0,$ where $q$ is an arbitrary real polynomial with $\deg(q) < \deg (p)$ and $q, p$ has no common zeros.
\begin{proof}[Proof of Theorem \ref{WTDD}]

Let $p$ and $p_i$ be the polynomials of the form \eqref{general polynomial form}, that is,
\begin{align*}
p(z)&=\prod\limits_{j=1}^m(z-\alpha_j)^{b_j},\,\,p_i(z)=\prod\limits_{j=1,j\neq i}^m(z-\alpha_j)^{b_j},\,i=1,2,\ldots,m.
\end{align*}
\begin{lem}
\begin{align*}
F_t[\underbrace{\alpha_1,\ldots,\alpha_1}_{b_1\text{\,times}},\underbrace{\alpha_2,\ldots,\alpha_2}_{b_2\text{\,times}},\ldots,\underbrace{\alpha_m,\ldots,\alpha_m}_{b_m\text{\,times}}]&= \sum\limits_{i=1}^m\bigg(\sum\limits_{j=1}^{b_i}\frac{(\log(1/t))^{j-1}}{(b_i-j)!(j-1)!}\bigg(\frac{q(z)}{p_i(z)}\bigg)^{(b_i-j)}(\alpha_i)\bigg)t^{-\alpha_i-1}.
\end{align*}
\end{lem}
\begin{proof}
Since $F_t(z)=q(z)t^{-z-1},$ we have that
\begin{align}\label{Ftsimple form}
F_t[\alpha_1,\alpha_2,\ldots,\alpha_m]&= \sum_{i=1}^m\frac{q(\alpha_i)}{P_i(\alpha_i,\alpha)}t^{-\alpha_i-1}.
\end{align}
First note that $P_i(\alpha_i,\alpha)=\prod_{r=1,r\neq i}^m(\alpha_i-\alpha_r)$ for $i=1,2,\ldots,m.$  It follows that
\begin{align*}
\frac{\partial^{k}}{\partial \alpha_1^{k}} \frac{q(\alpha_1)}{P_1(\alpha_1,\alpha)}&= \frac{\partial^{k}}{\partial z^{k}}\frac{q(z)}{P_1(z,\alpha)} |_{z=\alpha_1}= \bigg(\frac{q(z)}{P_1(z,\alpha)}\bigg)^{(k)}(\alpha_1),\\
\frac{\partial^{k}}{\partial \alpha_1^{k}} \frac{q(\alpha_i)\,t^{-\alpha_i-1}}{P_i(\alpha_i,\alpha)}&=\frac{k! \,\,q(\alpha_i)\,t^{-\alpha_i-1}}{(\alpha_i-\alpha_1)^{k+1}\prod_{r=2,r\neq i}^m(\alpha_i-\alpha_r)},\,\,i=2,3,\ldots,m.
\end{align*}
Using Leibniz's  rule we obtain
\begin{align*}
\frac{1}{k!}\frac{\partial^{k}}{\partial \alpha_1^{k}}\bigg(\frac{q(\alpha_1)}{P_1(\alpha_1,\alpha)}t^{-\alpha_1-1}\bigg)&=\frac{1}{k!} \sum\limits_{j=0}^{k} \binom{k}{j} \bigg(\frac{q(z)}{P_1(z,\alpha)}\bigg)^{(k-j)}(\alpha_1)\,\,(\log(1/t))^{j} \,t^{-\alpha_1-1},\\
&= \sum\limits_{j=0}^{k} \frac{(\log(1/t))^{j}}{(k-j)!j!} \bigg(\frac{q(z)}{P_1(z,\alpha)}\bigg)^{(k-j)}(\alpha_1)\, t^{-\alpha_1-1},\\
&= \sum\limits_{j=1}^{k+1} \frac{(\log(1/t))^{j-1}}{(k+1-j)!(j-1)!} \bigg(\frac{q(z)}{P_1(z,\alpha)}\bigg)^{(k+1-j)}(\alpha_1)\, t^{-\alpha_1-1}.
\end{align*}
Thus differentiating the both side of \eqref{Ftsimple form} w.r.t the variable $\alpha_1$ we get that,
\begin{multline*}
F_t[\underbrace{\alpha_1,\ldots,\alpha_1}_{b_1\text{\,times}},\alpha_2,\alpha_3,\ldots,\alpha_m] = \frac{1}{(b_1-1)!}\frac{\partial^{(b_1-1)}}{\partial \alpha_1^{(b_1-1)}}F_t[\alpha_1,\alpha_2,\ldots,\alpha_m] \\
= \sum\limits_{j=1}^{b_1} \frac{(\log(1/t))^{j-1}}{(b_1-j)!(j-1)!} \bigg(\frac{q(z)}{P_1(z,\alpha)}\bigg)^{(b_1-j)}(\alpha_1)\,t^{-\alpha_1-1} 
+\sum\limits_{i=2}^m \frac{q(\alpha_i)\,t^{-\alpha_i-1}}{(\alpha_i-\alpha_1)^{b_1}\prod_{r=2,r\neq i}^m(\alpha_i-\alpha_r)}.
\end{multline*}
Observe that the differentiation in $\alpha_2$ of each term in the above summation are as follows,
\begin{multline*}
\frac{1}{k!}\frac{\partial^{k}}{\partial \alpha_2^{k}} \bigg(\frac{q(z)}{P_1(z,\alpha)}\bigg)^{(b_1-j)}(\alpha_1)= \frac{1}{k!}\frac{\partial^{k}}{\partial \alpha_2^{k}}\frac{\partial^{(b_1-j)}}{\partial \alpha_1^{(b_1-j)}}\bigg(\frac{q(\alpha_1)}{P_1(\alpha_1,\alpha)}\bigg)=\frac{1}{k!}\frac{\partial^{(b_1-j)}}{\partial \alpha_1^{(b_1-j)}}\frac{\partial^{k}}{\partial \alpha_2^{k}}\bigg(\frac{q(\alpha_1)}{P_1(\alpha_1,\alpha)}\bigg),\\
= \frac{\partial^{(b_1-j)}}{\partial \alpha_1^{(b_1-j)}}\bigg(\frac{q(\alpha_1)}{(\alpha_1-\alpha_2)^{(k+1)}\prod\limits_{r=3}^{m}(\alpha_1-\alpha_r)}\bigg)= \bigg(\frac{q(z)}{(z-\alpha_2)^{(k+1)}\prod\limits_{r=3}^{m}(z-\alpha_r)}\bigg)^{(b_1-j)}(\alpha_1) ,
\end{multline*}

\begin{multline*}
\frac{1}{k!}\frac{\partial^{k}}{\partial \alpha_2^{k}}\bigg( \frac{q(\alpha_2)}{(\alpha_2-\alpha_1)^{b_1}\prod_{r=3}^m(\alpha_2-\alpha_r)}t^{-\alpha_2-1}\bigg) \\
= \sum\limits_{j=1}^{k+1} \frac{(\log(1/t))^{j-1}}{(k+1-j)!(j-1)!} \bigg(\frac{q(z)}{(z-\alpha_1)^{b_1}\prod_{r=3}^m(z-\alpha_r)}\bigg)^{(k+1-j)}(\alpha_2)\,t^{-\alpha_2-1},
\end{multline*}
and
\begin{multline*}
\frac{1}{k!}\frac{\partial^{k}}{\partial \alpha_2^{k}}\bigg( \frac{q(\alpha_i)}{(\alpha_i-\alpha_1)^{b_1}\prod_{r=2,r\neq i}^m(\alpha_i-\alpha_r)}t^{-\alpha_i-1}\bigg) \\
= \frac{q(\alpha_i)}{(\alpha_i-\alpha_1)^{b_1}(\alpha_i-\alpha_2)^{(k+1)}\prod_{r=3,r\neq 3}^m(\alpha_i-\alpha_r)}t^{-\alpha_i-1},\,\,i=3,4,\ldots,m.
\end{multline*}

Hence we get that 
\begin{multline*}
F_t[\underbrace{\alpha_1,\ldots,\alpha_1}_{b_1\text{\,times}},\underbrace{\alpha_2,\ldots,\alpha_2}_{b_2\text{\,times}},\alpha_3,\alpha_4,\ldots,\alpha_m]=
\frac{1}{(b_1-1)!(b_2-1)!}\frac{\partial^{(b_2-1)}}{\partial \alpha_2^{(b_2-1)}}\frac{\partial^{(b_1-1)}}{\partial \alpha_1^{(b_1-1)}}F_t[\alpha_1,\alpha_2,\ldots,\alpha_m],
\\
= \sum\limits_{j=1}^{b_1} \frac{(\log(1/t))^{j-1}}{(b_1-j)!(j-1)!} \bigg(\frac{q(z)}{(z-\alpha_2)^{b_2}\prod\limits_{r=3}^{m}(z-\alpha_r)}\bigg)^{(b_1-j)}(\alpha_1)\,t^{-\alpha_1-1}
\\
+\sum\limits_{j=1}^{b_2} \frac{(\log(1/t))^{j-1}}{(b_2-j)!(j-1)!}\bigg(\frac{q(z)}{(z-\alpha_1)^{b_1}\prod\limits_{r=3}^m(z-\alpha_r)}\bigg)^{(b_2-j)}(\alpha_2)\,t^{-\alpha_2-1}\\
+\sum\limits_{i=3}^m\frac{q(\alpha_i)}{(\alpha_i-\alpha_1)^{b_1}(\alpha_i-\alpha_2)^{b_2}\prod\limits_{r=3,r\neq i}^m(\alpha_i-\alpha_r)}t^{-\alpha_i-1}.
\end{multline*}
Continuing this way it is straightforward to see that
\begin{align*}
F_t[\underbrace{\alpha_1,\ldots,\alpha_1}_{b_1\text{\,times}},\underbrace{\alpha_2,\ldots,\alpha_2}_{b_2\text{\,times}},\ldots,\underbrace{\alpha_m,\ldots,\alpha_m}_{b_m\text{\,times}}]&= \sum\limits_{i=1}^m\bigg(\sum\limits_{j=1}^{b_i}\frac{(\log(1/t))^{j-1}}{(b_i-j)!(j-1)!}\bigg(\frac{q(z)}{p_i(z)}\bigg)^{(b_i-j)}(\alpha_i)\bigg)\,t^{-\alpha_i-1}.
\end{align*}
\end{proof}

This gives us that the weight function in $\eqref{weight function another form}$ can be written in the following form
\begin{align}\label{weight function and divided diff}
w_{q,p}(t)&= F_t[\underbrace{\alpha_1,\ldots,\alpha_1}_{b_1\text{\,times}},\underbrace{\alpha_2,\ldots,\alpha_2}_{b_2\text{\,times}},\ldots,\underbrace{\alpha_m,\ldots,\alpha_m}_{b_m\text{\,times}}],
\end{align}
where the family of function $\{F_t:t\in(0,1]\}$ is given by $F_t(z)=q(z)t^{-z-1},\,\, \Re(z)< 0.$
\end{proof}
Let us denote the weight function $w_{q,p}(t)$ by $w_p(t)$ in the case when $q\equiv 1.$ In such case we provide an alternative expression for the weight function $w_{p}(t)$ in terms of convolution of some exponential functions. Let $\alpha_1,\alpha_2,\ldots,\alpha_m$ be an arbitrary $m$ points (not necessarily distinct) in $\mathbb H_{-}.$  Let $p$ be the polynomial given by
\begin{align*}
p(z)&=\prod\limits_{j=1}^m(z-\alpha_j). 
\end{align*} 

Laplace transform of a function $f:[0,\infty)\to \mathbb C$ is defined by
\begin{align*}
(\mathcal Lf)(s)&= \int_0^{\infty}\exp(-sy)f(y)dy,
\end{align*} for those $s\in\mathbb C$ for which the integral make sense. Consider functions
\begin{align*}
f_j(y)&= \exp(\alpha_j y), \,\,\,j=1,2,\ldots,m.
\end{align*} Their Laplace transforms are
\begin{align*}
(\mathcal Lf)(s)&= \frac{1}{s-\alpha_j},\,\,\,\,j=1,2,\ldots,m,\,Re(s)>0.
\end{align*}
The convolution of two functions $g$ and $h$ is defined by
\begin{align*}
(g*h)(y)&=\int_0^y g(x)h(y-x)dx,
\end{align*} where $g,h:[0,\infty)\to \mathbb C$ are two complex valued function so that the integrals mentioned above are finite for all $y\in[0,\infty)$. The Laplace transform of $g*h$ is $$(\mathcal L(g*h))(s) = (\mathcal Lg)(s)(\mathcal Lh)(s).$$

Let $f$ be the function defined by $f= f_1*f_2*\cdots*f_m.$ Then it follows that
\begin{align*}
\frac{1}{p(s)}&=\int_{0}^{\infty} \exp(-sy)f(y)dy =\int_{0}^{1}t^{s-1}f(\log(1/t))dt,\,\,\,s>0.
\end{align*}
Thus the weight function has the following form 
\begin{align}\label{simplex form for weight}
w_{p}(t)= \frac{1}{t}(f_1*f_2*\cdots*f_m)(\log(1/t)), \,\,\,t\in(0,1].
\end{align}
Since each $f_j$ is a exponential function, the convolution of those $f_j$'s has the following simple expression.

\begin{align}\label{convolution expression}
(f_1*f_2*\cdots*f_m)(y)&= y^{m-1}\int_{\Delta_{m-1}}\exp\big(y \sum _{j=1}^m \lambda_j\alpha_j\big)  d\lambda,
\end{align} where $\Delta_{m-1}= \{(\lambda_1,\lambda_2,\ldots,\lambda_{m})\in [0,1]^{m}:\sum _{j=1}^m \lambda_j=1\}$ is the $(m-1)$ dimensional simplex and $d\lambda$ is the usual Lebesgue measure on the simplex $\Delta_{m-1}.$ 

\begin{proof}[Proof of \eqref{convolution expression}]
{\bf{$m=2:$}}
\begin{align*}
(f_1*f_2)(y)&= \int_0^y \exp(\alpha_1 x)\exp(\alpha_2(y- x))dx = y\int_0^1 \exp(\alpha_1 sy)\exp(\alpha_2(y- sy))ds,\\
&= y\int_0^1 \exp(y (s\alpha_1 + (1-s)\alpha_2))ds = y \int_{\Delta_{1}}\exp\big(y \sum _{j=1}^2 \lambda_j\alpha_j\big)  d\lambda ,
\end{align*}
Now we use induction to prove for an arbitrary $m.$ Let us assume that the statement is true for $m=k.$ So by definition of integral over the simplex  $\Delta_{k-1}$ we have
\begin{align*}
(f_1*f_2*\cdots*f_{k})(x)&= x^{k-1}\int_{0}^1 \int_{0}^{1-\lambda_1}\ldots \int_{0}^{1-\lambda_1-\cdots-\lambda_{k-2}}\exp\big(x \sum _{j=1}^k \lambda_j\alpha_j\big)  d\lambda_1\cdots d\lambda_{k-1},
\end{align*}
where $\lambda_k=1-\lambda_1-\cdots-\lambda_{k-1}.$   So, $(f_1*f_2*\cdots*f_{k+1})(y)=\int_0^y(f_1*f_2*\cdots*f_{k})(x) f_{k+1}(y-x)dx.$ We apply a change of variable $x=sy$ to obtain that
\begin{align*}
g(y)&= (f_1*f_2*\cdots*f_{k+1})(y)= y \int_0^1(f_1*f_2*\cdots*f_{k})(sy) \cdot f_{k+1}(y-sy)ds,\\
 &= y^{k}\int_{0}^1 s^{k-1}\int_{0}^1 \int_{0}^{1-\lambda_1}\ldots \int_{0}^{1-\lambda_1-\cdots-\lambda_{k-2}}\exp\bigg(y \big(\sum _{j=1}^k s\lambda_j\alpha_j + (1-s)\alpha_{k+1}\big) \bigg)  d\lambda_1\cdots d\lambda_{k-1} ds.
\end{align*}
Now we again apply a change of variable, namely $s\lambda_j=c_j$ for $j=1,2,\dots,k-1,$ and we obtain that
\begin{align*}
g(y)&= y^{k}\int_{0}^1\int_{0}^{s} \int_{0}^{s-c_1}\ldots \int_{0}^{s-c_1-\cdots-c_{k-2}}\exp\bigg(y \big(\sum _{j=1}^k c_j\alpha_j + (1-s)\alpha_{k+1}\big) \bigg)  dc_1\cdots dc_{k-1} ds,
\end{align*}
where $c_k= s-c_1-\cdots-c_{k-1}.$ Finally we apply the change of variable $s=1-c_0$ and we get that
\begin{align*}
g(y)&= y^{k}\int_{0}^1\int_{0}^{1-c_0} \int_{0}^{1-c_0-c_1}\ldots \int_{0}^{1-c_0-c_1-\cdots-c_{k-2}}\exp\bigg(y \big(c_0\alpha_{k+1}+\sum _{j=1}^{k} c_j\alpha_j\big) \bigg) dc_1\cdots dc_{k-1} dc_0,
\end{align*} 
where $c_k$ satisfies the relationship $c_k= 1-\sum_{j=0}^{k-1}c_j.$ Hence it follows that
\begin{align*}
(f_1*f_2*\cdots*f_{k+1})(y)= y^{k}\int_{\Delta_{k}}\exp\big(y \sum _{j=1}^{k+1} \lambda_j\alpha_j\big)  d\lambda,
\end{align*} 
where $\Delta_{k}= \{(\lambda_1,\lambda_2,\ldots,\lambda_{k+1})\in [0,1]^{m}:\sum _{j=1}^{k+1} \lambda_j=1\}$ is the $k$ dimensional simplex and $d\lambda$ is the usual Lebesgue measure on the simplex $\Delta_{k}.$

\end{proof}

Hence using \eqref{simplex form for weight} we get that the weight function $w_{p}(t)$ has the following form 
\begin{align}\label{weight and simplex}
w_{p}(t) &= \frac{1}{t} (\log(1/t))^{m-1}\int_{\Delta_{m-1}}\big(\frac{1}{t}\big)^ {\sum _{j=1}^m \lambda_j\alpha_j}  d\lambda,\,\,\,t\in(0,1].
\end{align}

As a corollary we observe that the Hausdorff moment property is preserved when the roots are shifted or rescaled; it could be used
to prove Theorem \ref{Sameer MTP} below and simplify 
some of our proofs. This is  also consequence of the known
relation between the Hausdorff moment sequence and completely
monotone functions on $[0,\infty)$, see the paragraph
after Theorem \ref{WTDD}. Here we find relationship between the weight function 
for the shifted and rescaled  rational function. More precisely, 
\begin{lem}\label{shifted and scaling}
Let $p$ be a polynomial with real coefficient having all its roots in $\mathbb H_{-}.$ Let $\tilde{p}(z)= p(z+c)$ and $\hat{p}(z)= p(z/d),$ where $c$ is any real number such that roots of $\tilde{p}$ lies in $\mathbb H_{-}$ and $d>1.$ Then it follows that if the sequence $(1/p(n))_{n\in\mathbb Z_+}$ is a Hausdorff moment sequence, then both the sequences $(1/\tilde{p}(n))_{n\in\mathbb Z_+}$ and $(1/\hat{p}(n))_{n\in\mathbb Z_+}$ is also Hausdorff moment sequence with weight functions given by
\begin{align*}
w_{\tilde{p}}(t)&= \frac{1}{t^c}w_{p}(t),\,\,w_{\hat{p}}(t)=\frac{1}{b} (\log(1/t))^{m-1}(\log (b/t))^{1-m} w_{p}(t/b)\,\,\,\,t\in(0,1].
\end{align*}
\end{lem}


\section{A necessary condition}
In this section we start with the proof of Theorem \ref{not moment criterion} which gives us a necessary condition on the zeros of $p$ so that $(q(n)/p(n))_{n\in\mathbb Z_+}$ is a Hausdorff moment sequence.
%

\begin{proof}[Proof of Theorem \ref{not moment criterion}]
We have $x_1>r_i$ for $i=1,2,\ldots,s$ and $x_1>x_i$ for $i=2,3,\ldots,d.$ Earlier we obtained that $\frac{q(n)}{p(n)}= \int_0^1 t^nw_{q,p}(t)dt,$ where the weight function $w_{q,p}(t)$ is given by 
\begin{align*}
w_{q,p}(t)= \sum\limits_{i=1}^s\big(\sum\limits_{j=1}^{l_i}\frac{A_i^j(\log(1/t))^{j-1}}{(j-1)!}\big)t^{-r_i-1} + \sum\limits_{i=1}^d\big(\sum\limits_{j=1}^{m_i}\frac{|A_{s+i}^j|(\log(1/t))^{j-1}2\cos(\theta_i^j-y_i\log t)}{(j-1)!}  \big)t^{-x_i-1}. 
\end{align*}
The sequence $q(n)/p(n)$ is a Hausdorff moment sequence if and only if $w_{q,p}(t)\geq 0$ for all $t\in (0,1].$ It is straightforward to see that $t^{x_1+1}w_{q,p}(t)$ will take the following form
\begin{multline*}
t^{x_1+1}w_{q,p}(t)= \sum\limits_{i=1}^s\big(\sum\limits_{j=1}^{l_i}\frac{A_i^j(\log(1/t))^{j-1}}{(j-1)!}\big)t^{x_1-r_i} + \sum\limits_{j=1}^{m_1}\frac{|A_{s+1}^j|(\log(1/t))^{j-1}2\cos(\theta_1^j-y_1\log t)}{(j-1)!}  \\
+ \sum\limits_{i=2}^d\big(\sum\limits_{j=1}^{m_i}\frac{|A_{s+i}^j|(\log(1/t))^{j-1}2\cos(\theta_i^j-y_i\log t)}{(j-1)!}  \big)t^{x_1-x_i}.
\end{multline*}
Our aim is to show that the function $t^{x_1+1}w_{q,p}(t)$ is strictly negative for some  $t_0$ in $(0,1].$ For the sake of simplicity we write $t^{x_1+1}w_{q,p}(t)= f_1(t)+f_2(t),\,t\in(0,1],$ where $f_1$ and $f_2$ is given by 
\begin{align*}
f_1(t) &= \sum\limits_{i=1}^s\big(\sum\limits_{j=1}^{l_i}\frac{A_i^j(\log(1/t))^{j-1}}{(j-1)!}\big)t^{x_1-r_i} + \sum\limits_{i=2}^d\big(\sum\limits_{j=1}^{m_i}\frac{|A_{s+i}^j|(\log(1/t))^{j-1}2\cos(\theta_i^j-y_i\log t)}{(j-1)!}  \big)t^{x_1-x_i},\\
f_2(t) &= \sum\limits_{j=1}^{m_1}\frac{|A_{s+1}^j|(\log(1/t))^{j-1}2\cos(\theta_1^j-y_1\log t)}{(j-1)!}.
\end{align*}
For any positive number $c>0$ and $j\in\mathbb N,$ we have $t^c (\log(1/t))^j\to 0$ as $t\to 0.$ It follows that $f_1(t)\to 0$ as $t\to 0.$ Consider the sequence $t_k= \exp(\frac{\theta_1^{m_1}-(2k+1)\pi}{y_1}),\,\,k\in\mathbb N,$ so that $\cos (\theta_1^{m_1}-y_1\log t_k)= -1.$ Note that $t_k\to 0$ as $k\to \infty.$ 
\begin{align*}
f_2(t_k)= \big(\sum\limits_{j=1}^{m_1-1}\frac{|A_{s+1}^j|(\log(1/t_k))^{j-1}2\cos(\theta_1^j-y_1\log t_k)}{(j-1)!}\big) -2 \frac{|A_{s+1}^{m_1}|(\log(1/t_k))^{{m_1}-1}}{(m_1-1)!}.
\end{align*}
Now two cases arises, namely $m_1=1$ and $m_1>1.$

{\bf{Case 1}: $m_1=1.$} In this case $f_2(t_k)=-2|A_{s+1}^{1}|<0.$ Since $f_1(t_k)\to 0$ as $k\to \infty,$ it follows that there exists $k_0\in\mathbb N$ such that ${t_k}^{x_1+1}w_{q,p}(t_k) = f_1(t_k)+f_2(t_k)$ is strictly negative for all $k\geq k_0.$ By continuity $w_{q,p}(t) <0$ in some interval containing $t_{k_0}.$ Thus the measure $ w_{q,p}(t) dt$ is not positive measure, and the sequence $q(n)/p(n)$ is not a Hausdorff moment sequence.

{\bf{Case 2}: $m_1>1.$} Consider the real polynomial  $u(x)$ defined by
\begin{align*}
u(x)=\big(\sum\limits_{j=1}^{m_1-1}2\frac{|A_{s+1}^j|}{(j-1)!}x^{j-1}\big) -2 \frac{|A_{s+1}^{m_1}|}{(m_1-1)!}x^{{m_1}-1},\,\,x\in\mathbb R.
\end{align*}
 Since $A_{s+1}^{m_1}\neq 0,$ $u(x)\to - \infty$ as $x\to\infty.$ Thus for an arbitrary large $M>0,$ there exists a $x_0\in \mathbb R$ such that $u(x)< -M$ for all $x > x_0.$ Note that $f_2(t)< u(\log 1/t)$ for all $t\in (0,1]$ and we have $\log (1/t_k)\to \infty$ as $k\to \infty.$ Consequently there exists $k_0\in\mathbb N$ such that $f_2(t_k)< u(\log 1/t_k) < -M$ for all $k\geq k_0.$ Since $f_1(t_k)\to 0$ as $k\to\infty,$ it follows that there exist $N\in\mathbb N$ such that ${t_k}^{x_1+1}w_{q,p}(t_k) = f_1(t_k)+f_2(t_k)$ is strictly negative for $k\geq N.$ Hence $w_{q,p}(t) <0$ in some interval containing $t_{N}.$ Thus the sequence $q(n)/p(n)$ is not a Hausdorff moment sequence.
\end{proof}

We use this result to answer the question of G. Misra. We show
that the Schur product $SK(z, w)  
$
of a subnormal kernel $K$ with the S\"{z}ego kernel $S(z, w)=(1-z\bar w)^{-1}$ on the unit disc needs not necessarily be a subnormal kernel.

Let us consider the following family of kernel functions:
\begin{align*}
K_c(z,w) &= \sum\limits_{n=0}^{\infty} (n+c)^6 (z\bar{w})^n,\,\,c>0,\,z,w\in\mathbb D. 
\end{align*}
It is straightforward to see that for every $c>0,$ the sequence $\{(n+c)^{-6}\}_{n\in\mathbb Z_+}$ is a Hausdorff moment sequence and associated with the measure $\frac{1}{5!}t^{c-1}(\log (1/t))^{5}dt$ on the unit interval $(0,1].$ Hence the kernel function $K_c$ is a contractive and subnormal kernel. The Schur product of $K_c$ with the S\"{z}ego kernel $S(z, w)=(1-z\bar w)^{-1}=\sum (z\bar w)^n$
on the unit disc $\mathbb D$ is given by
\begin{align*}
 S K_c(z,w) 
  &= \sum\limits_{n=0}^{\infty} \big(\sum\limits_{j=0}^{n}(j+c)^6 \big)(z\bar{w})^n,\,\,\,z,w\in\mathbb{D}.
\end{align*}
Let $p_c(n)= \sum_{j=0}^{n} (j+c)^6.$ Since the kernel $SK_c$ is already a contractive kernel, it follows that the kernel function $SK_c$ is subnormal if and only if the sequence $(1/p_c(n))_{n\in\mathbb Z_+}$ is a Hausdorff moment sequence. Using
the known formulas
for the power sums $\sum_{j=0}^{n} j^k$, it is easy to see that for every $c>0$ the term $p_c(n)$ is a real polynomial in $n$ of degree $7.$
For $c=1$ we use Faulhaber's formula for sums of powers of integers (see \cite{Beardon}) to obtain 
\begin{align*}
  p_1(n) &= \sum\limits_{j=0}^{n} (j+1)^6=
           \frac{(n+1)(n+2)(2n+3) 
           \big(3 (n+1)^4+6(n+1)^3-3(n+1)+1\big)}{42}
\end{align*}
a polynomial of degree $7$. That is,  $p_1(z) =
(z+1)(z+2)(2z+3) q(z)$ has three real roots
$\{-2, -\frac{3}{2}, -1\}$, with $q(z)=  \frac     {3 (z+1)^4+6(z+1)^3-
  3(z+1)+1}{42}$ a polynomial
of degree $4$.
All roots of any polynomial of degree $4$ can be found explicitly, in particular we can get all roots of $q$ and then of $p_1.$
We perform numerically computations instead and find that  the complex roots are approximately given by $\{-0.62\pm 0.16i, -2.38\pm 0.16i\}.$
As $\Re(-0.62\pm 0.16i)=
-0.62>-1, -\frac 23, -2$, it  follows from Theorem $\ref{not moment criterion}$ that there exists a $t_0\in (0, 1)$ such
that the weight function $w_{p_1}(t_0) <0$
and that $\{1/p_1(n)\}_{n\in\mathbb Z_+}$ is not a Hausdorff moment sequence. By continuity in $c$, it follows that the weight function
$w_{p_c}(t_0)<0$ for all $c$ in some neighborhood of $c=1.$ Consequently $\{1/p_c(n)\}_{n\in\mathbb Z_+}$ is not a Hausdorff moment sequence for all $c$ in some neighborhood of $1.$ Thus we produce a family of subnormal kernel $K$ for which $SK$ is not a subnormal.

Now we will discuss few special cases of rational function $q(z)/p(z).$ First let us assume that all the roots of $p$ are real and negative. Let $p$ be the polynomial of the form
\begin{align*}
p(x)=\prod\limits_{j=1}^{m}(x-\alpha_j)^{b_j},
\end{align*}where $\alpha_j$'s are distinct negative real number. We order the roots so that $\alpha_{k}<\alpha_{k-1},$ for $k=2,3,\ldots,m$ and $\alpha_1<0.$ Let $q$ be any polynomial with real coefficients such that $\deg(q)< \deg(p).$ In this case, using $\eqref{weight function another form}$ we have, 

\begin{align*}
w_{q,p}(t) &= \sum\limits_{i=1}^m\bigg(\sum\limits_{j=1}^{b_i}\frac{A_i^j(\log(1/t))^{j-1}}{(j-1)!}\bigg)t^{-\alpha_i-1}, \,\,\,t\in (0,1],
\\\text{where\,\,}\,\,\,A_i^j&= \frac{\big(\tfrac{q}{p_i}\big)^{(b_i-j)}(\alpha_i)}{(b_i-j)!},\,\,i=1,2,\ldots,m,\,\,j=1,2,\ldots,b_i.
\end{align*} 
So, in this case $t^{\alpha_1+1}w_{q,p}(t)$ takes the following form
\begin{align*}
t^{\alpha_1+1}w_{q,p}(t)&= \sum\limits_{j=1}^{b_1}\frac{A_1^j(\log(1/t))^{j-1}}{(j-1)!}+\sum\limits_{i=2}^m\bigg(\sum\limits_{j=1}^{b_i}\frac{A_i^j(\log(1/t))^{j-1}}{(j-1)!}\bigg)t^{\alpha_1-\alpha_i}\\
&= f_1(t)+f_2(t),\,\,\,t\in(0,1],
\end{align*}where $f_1$ and $f_2$ is given by
\begin{align*}
f_1(t)&= \sum_{j=1}^{b_1}\frac{A_1^j(\log(1/t))^{j-1}}{(j-1)!},\,\,\,f_2(t)=\sum\limits_{i=2}^m\bigg(\sum\limits_{j=1}^{b_i}\frac{A_i^j(\log(1/t))^{j-1}}{(j-1)!}\bigg)t^{\alpha_1-\alpha_i},\,\,t\in(0,1].
\end{align*}  
Since $f_1(t)$ is a polynomial in $\log(1/t)$ and $A_1^{b_1}=\tfrac{q}{p_1}(\alpha_1)\neq 0,$ it follows that $f_1(t)\to +\infty$ as $t\to 0$ if $A_1^{b_1}>0$ and $f_1(t)\to -\infty$ as $t\to 0$ if $A_1^{b_1}<0$. Also we have $f_2(t)\to 0$ as $t\to 0.$ So, if $w_{q,p}(t)\geq 0$ on $(0,1],$ then necessarily we must have $A_1^{b_1}>0,$ that is $q(\alpha_1)>0.$ Thus we get a necessary condition, in this case, for $\{q(n)/p(n)\}_{n\in\mathbb Z_+}$ to be a Hausdorff moment sequence.
\begin{thm}\label{rational necessary}
Let $p$ be a polynomial with all negative real roots. Let $\alpha_1$ be the largest among all roots of $p.$ Let $q$ be another polynomial with real coefficients such that $\deg(q)< \deg(p)$ and have no common roots with $p.$ If $\{q(n)/p(n)\}_{n\in\mathbb Z_+}$ is a Hausdorff moment sequence, then necessarily we have $q(\alpha_1)>0.$ This condition is also sufficient if $q$ is a monic polynomial of degree $1.$ 
\end{thm} 

The sufficient part follows from the following observation.
Let $q$ be the polynomial $q(x)= (x-\beta)$ for some $\beta\in\mathbb R.$ The necessity condition $q(\alpha_1)>0$ gives us $\beta < \alpha_1.$ In that case 
\begin{align*}
\frac{n-\beta}{n-\alpha_1}&= 1+\frac{\alpha_1-\beta}{n-\alpha},\,\,\,n\in\mathbb Z_+.
\end{align*}
This gives us that the sequence $\{\frac{n-\beta}{n-\alpha_1}\}_{n\in\mathbb Z_+}$ is a Hausdorff moment sequence. Now it is straightforward to see that the sufficiency part of the Theorem $\ref{rational necessary}$ follows using the Lemma $\ref{Berg Duran product theorem}.$

We have obtained an expression for the weight function $w_{q,p}$ in terms of finite divided differences in $\eqref{weight function and divided diff}.$ This expression can be used to provide a sufficient condition on $q$ so that the sequence $\{q(n)/p(n)\}_{n\in\mathbb Z_+}$ becomes a Hausdorff moment sequence.

\begin{proof}[Proof of Proposition \ref{sufficient for rational}]
Using $\eqref{weight function and divided diff},$ in this case, we have
\begin{align*}
w_{q,p}(t)&= F_t[\alpha_1,\alpha_2,\ldots,\alpha_m],\,\,\,t\in(0,1],
\end{align*}where $F_t(z)=q(z)t^{-z-1},$ for $\Re(z) <0.$ Now we apply Leibniz's rule for finite divided difference of product of two functions to obtain
\begin{align*}
w_{q,p}(t)&= \sum\limits_{j=1}^m \big(q(z)[\alpha_1,\alpha_2,\ldots,\alpha_j]\big)\big(t^{-z-1}[\alpha_j,\alpha_{j+1},\ldots,\alpha_m]\big),\,\,\,\,t\in(0,1].
\end{align*}
Since $\alpha_j$'s are all real, applying mean value theorem for finite divided differences, we obtain that there exists a $\zeta_j$ in $\mathbb R,$ determined by $\{\alpha_j,\ldots,\alpha_m\}$ so that
\begin{align*}
t^{-z-1}[\alpha_j,\alpha_{j+1},\ldots,\alpha_m]&= \frac{1}{(m-j)!} \frac{\partial^{m-j}}{\partial z^{m-j}} t^{-z-1}(\zeta_j),\\
&= \frac{1}{(m-j)!}t^{\zeta_j-1} (\log(1/t))^{m-j} \geq 0,\,\,\,t\in(0,1].
\end{align*}
Thus $w_{q,p}(t)\geq 0$ for all $t\in(0,1]$ if $q[\alpha_1,\alpha_2,\ldots,\alpha_j]\geq 0$ for $j=1,2,\ldots,m.$
\end{proof}

Let $r$ be a rational function of the form $q(x)/p(x),$ where $\deg(q)=\deg(p)$ and assume that all the roots of $q,p$ are distinct negative real number. Ball provided a sufficient condition on the zeros and poles of $r$ so that $r(x)$ is a completely monotone function on $[0,\infty)$, see \cite{ball94cmrf}.  In particular these conditions ensure that $(r(n))_{n\in\mathbb Z_+}$ is a Hausdorff moment sequence. Using Proposition \ref{sufficient for rational}, we obtain an alternative proof of the result of Ball in the case when $\deg(q)=\deg(p)$ is at most $3.$ Let $r$ be a rational function given by 
\begin{align*}
r(x)&= \tfrac{(x+b_1)(x+b_2)}{(x+p_1)(x+p_2)},\,\,\,\,0<b_1<b_2, \,\,0<p_1<p_2,
\end{align*}
We will show that if $p_1<b_1$ and $p_1+p_2< b_1+b_2$ then $(r(n))_{n\in\mathbb Z_+}$ is a Hausdorff moment sequence.

Let $q,p$ be the polynomial given by $q(x)=(x+b_1)(x+b_2)$ and $p(x)=(x+p_1)(x+p_2).$ From our assumption we get that the divided differences of $(q-p)$ 
\begin{align*}
(q-p)[-p_1]&= (b_1-p_1)(b_2-p_1)>0,\,\,\,\, (q-p)[-p_1,-p_2]= (b_1+b_2)-(p_1+p_2) >0.
\end{align*}
Using Proposition \ref{sufficient for rational}, we get that
$r(n)-1=(q-p)(n)/p(n)$ is a Hausdorff moment sequence and
consequently $(r(n))_{n\in\mathbb Z_+}$ is a Hausdorff moment
sequence. In a similar manner, for a rational function $r(x)$ of the form 
\begin{align*}
r(x)&= \tfrac{(x+b_1)(x+b_2)(x+b_3)}{(x+p_1)(x+p_2)(x+p_3)},\,\,\,\,0<b_1<b_2<b_3, \,\,0<p_1<p_2<p_3,
\end{align*} 
it can be shown that if $p_1<b_1, p_1+p_2<b_1+b_2,p_1+p_2+p_3< b_1+b_2+b_3,$ then the divided differences $(q-p)[-p_1],(q-p)[-p_1,-p_2]$ and $(q-p)[-p_1,-p_2,-p_3]$ are strictly positive, where 
\begin{align*}
q(x)&=\prod\limits_{j=1}^3(x+b_j), \,\,\,p(x)=\prod\limits_{j=1}^3(x+p_j).
\end{align*}
 Consequently, we will have $(r(n))_{n\in\mathbb Z_+}$ is a Hausdorff moment sequence.
It might be interesting to find necessary and sufficient conditions for this class
of rational functions $r(x)$ so that 
$(r(n))_{n\in\mathbb Z_+}$ is a Hausdorff moment sequence.

\section{Sequences induced by polynomials}
In this section first we describe all polynomials $p$ with real coefficients up to degree $4$ and having roots in $\mathbb H_-$ so that the sequence $\{1/p(n)\}_{n\in\mathbb Z_+}$ is a Hausdorff moment sequence. As we have seen that if $p$ is reducible over $\mathbb R,$ then for such $p$ the sequence $\{1/p(n)\}_{n\in\mathbb Z_+}$ is always a Hausdorff moment sequence, see Lemma \ref{Berg Duran product theorem}. So, now on we consider only polynomials $p$ which are not reducible over $\mathbb R,$ that is, $p$ has at least one pair of non real roots.

The necessary condition in Theorem \ref{not moment criterion} gives us the answer in the case of degree $2$ polynomial $p.$ 

Let $p$ be a real polynomial of degree $3$ with a non real root, say $\alpha.$  Assume that $p$ is of the form
\begin{align*}
 p(z)&=(z-r)(z-\alpha)(z-\bar{\alpha}), 
\end{align*} where $r, \Re(\alpha) < 0.$  In the case of $r=-1,$ Theorem \ref{AC deg 3} shows that for such polynomial $p$ the sequence $\{1/p(n)\}_{n\in\mathbb Z_+}$ is a Hausdorff moment sequence if and only if $ \Re(\alpha) \leq 1.$  Their proof can be modified to get the answer for an arbitrary $r <0;$ it can also be proved using
our Lemma \ref{shifted and scaling} and Theorem \ref{not moment criterion}. We state this as a
\begin{thm}[Anand and Chavan]\label{Sameer MTP}
For $p(z)=(z-r)(z-\alpha)(z-\bar{\alpha})$ with $r, \Re(\alpha) < 0, \Im(\alpha)\neq 0,$ the sequence $\{1/p(n)\}_{n\in\mathbb Z_+}$ is a Hausdorff moment sequence if and only if $\Re(\alpha) \leq r.$  
\end{thm}

To get a complete classification for all real polynomial $p$ of degree $4$ for which $\{1/p(n)\}$ is a Hausdorff moment sequence, we need to consider  different cases separately.
\begin{thm}\label{degree 4}
Let $p$ be a real polynomial of the form  $p(z)= \prod\limits_{j=1}^2(z-(r+iy_j))\prod \limits _{j=1}^2(z-(r-iy_j)),$
where $r<0$ and $0< y_1 < y_2.$ Then the sequence $\{1/p(n)\}_{n\in\mathbb Z_+}$ is never a Hausdorff moment sequence.
\end{thm}
\begin{proof}
In this case the weight function $w_{1,p}(t)$ is given by 
\begin{align*}
w_{1,p}(t) &=t^{-r-1}\big(\frac{\sin (y_1\log t)}{y_1(y_1^2-y_2^2)}+ \frac{\sin (y_2\log t)}{y_2(y_2^2-y_1^2)}\big)\\
&= \frac{t^{-r-1}}{y_2(y_2^2-y_1^2)}\big( \frac{y_2}{y_1}\sin (y_1\log (1/t)) - \sin(\frac{y_2}{y_1}y_1\log(1/t)) \big)\,\,\,\,t\in(0,1].
\end{align*}
Since $\frac{y_2}{y_1}>1,$ we get that $w_{1,p}(t)<0$ whenever $y_1\log (1/t) = \frac{3\pi}{2}.$ Hence $\{1/p(n)\}$ is not a Hausdorff moment sequence.
\end{proof}

Thus in view of Theorem $\ref{Berg Duran product theorem},$ Theorem $\ref{not moment criterion},$ Theorem $\ref{Sameer MTP}$ and Theorem $\ref{degree 4},$  we obtain the following classification of all all real polynomials $p$ of degree $4$ for which $\{1/p(n)\}$ is a Hausdorff moment sequence.
\begin{thm}
Let $p$ be a real polynomial of degree $4$ having a non real root, say $\alpha.$ Assume all the roots of $p$ lies in $\mathbb H_-.$ Then the sequence $(1/p(n))_{n\in\mathbb Z_+}$ is a Hausdorff moment sequence if and only if there exists a real root, say $r,$ of $p$ such that $\Re(\alpha)\leq  r.$
\end{thm}

Although we could not classify all real polynomial $p$ of degree $5$ but we have a complete classification result for a special class of degree $5$ polynomials, namely polynomials whose roots lies in a vertical line. 
\begin{proof}[Proof of Theorem \ref{deg 5 special}]
We will divide the proof in the following two cases.

{\bf{Case 1:} $y_1=y_2:$} In this case we have $p(z)= (z-r)(z-(r+iy_1))^2(z-(r-iy_1))^2.$ Using partial decomposition formula of $1/p(z)$ from Proposition $\ref{PDF Theorem}$ we obtain
\begin{align*}
\frac{1}{p(z)} &= \frac{A_1}{z-r}+\frac{A_2}{z-(r+iy_1)}+ \frac{\overline{A_2}}{z-(r-iy_1)}+\frac{A_3}{(z-(r+iy_1))^2}+\frac{\overline{A_3}}{(z-(r-iy_1))^2},
\end{align*}
where $A_1= \frac{1}{y_1^4},A_2= -\frac{1}{2y_1^4}, A_3= \frac{i}{4y_1^3}.$ 

Thus from $\eqref{weight function}$ we get that 
$\frac{1}{p(n)}=\int_0^1 t^n w_p(t)dt,$ where the weight function $w_p(t)$ is given by
\begin{align*}
w_p(t)&= \frac{1}{y_1^4}t^{-r-1}- \frac{2\cos(y_1\log t)}{2y_1^4} t^{-r-1} + \frac{\log(\frac{1}{t})2\cos (\frac{\pi}{2}-y_1\log t)}{4y_1^3}t^{-r-1},\\
&= \frac{t^{-r-1}}{y_1^4}\big( 1- \cos (y_1\log t)-\frac{y_1\log t}{2}\sin(y_1\log t)\big).
\end{align*}
Consider the sequence $t_k= \exp (-\frac{(4k+1)\pi}{2y_1}),\,k\in\mathbb N.$ Note that  $w_p(t_k)= \frac{t_k^{-r-1}}{y_1^4}(1-\frac{(4k+1)\pi}{4})< 0$ for all $k\in \mathbb N.$ Hence the sequence $\frac{1}{p(n)}$ is not a Hausdorff moment sequence.

{\bf{Case 2:} $y_1<y_2:$} In this case, using partial decomposition formula of $1/p(z)$ from  Proposition $\ref{PDF Theorem}$ we obtain that
\begin{align*}
\frac{1}{p(z)} &= \frac{A_1}{z-r}+\frac{A_2}{z-(r+iy_1)}+ \frac{\overline{A_2}}{z-(r-iy_1)}+\frac{A_3}{z-(r+iy_2)}+\frac{\overline{A_3}}{z-(r-iy_2)},
\end{align*}
where $A_1= \frac{1}{y_1^2y_2^2},A_2= \frac{1}{2y_1^2(y_1^2-y_2^2)}, A_3= \frac{1}{2y_2^2(y_2^2-y_1^2)}.$ 

Again from  $\eqref{weight function}$ we get that 
$\frac{1}{p(n)}=\int_0^1 t^n w_p(t)dt,$ where the weight function $w_p(t)$ is given by
\begin{align*}
w_p(t)&= \frac{1}{y_1^2y_2^2}t^{-r-1}+ \frac{2\cos(y_1\log t)}{2y_1^2(y_1^2-y_2^2)} t^{-r-1} + \frac{2\cos (y_2\log t)}{2y_2^2(y_2^2-y_1^2)}t^{-r-1},\\
&= \frac{t^{-r-1}}{y_1^2y_2^2(y_2^2-y_1^2)}\big((y_2^2-y_1^2) - y_2^2\cos (y_1\log t)+ y_1^2 \cos (y_2\log t)\big),\\
&= \frac{t^{-r-1}}{y_2^2(y_2^2-y_1^2)}\big(\frac{y_2^2}{y_1^2}-1 - \frac{y_2^2}{y_1^2}\cos (y_1\log t)+  \cos (\frac{y_2}{y_1}y_1\log t)\big).
\end{align*}
Let $u=\frac{y_2}{y_1}$ so that $u>1.$ Our aim is to show that $w_p(t)\geq 0$ for all  $t\in (0,1]$ if and only $u$ is an integer greater than 1. We apply a change of variable $x=y_1\log (1/t).$ As $t\in (0,1],$ we have $x\in [0,\infty).$ Let $g$ be the function on $[0,\infty)$ defined by 
\begin{align*}
g(x)&=g_u(x)= u^2-1-u^2\cos(x)+\cos(ux),\,x\in[0,\infty).
\end{align*}

Then $w_p(t)= \frac{t^{-r-1}}{y_2^2(y_2^2-y_1^2)} g(y_1\log (1/t))$ for all $t\in (0,1].$ It is now sufficient to show that the function $g(x)\geq 0$ for all $x\in [0,\infty)$ if and only if $u$ is an integer. 

First observe that $g(2\pi)= \cos (2u\pi)-1.$ Hence $g(2\pi)< 0$ if $u$ is not an integer. 

Now let us assume $u$ is  a positive integer. We will show in this case $g(x)\geq 0$ for all $x\in [0,\infty).$ We have
\begin{align*}
g^{\prime} (x)&= u(u\sin(x)-\sin(ux)),\\
g^{\prime \prime}(x)&= u^2(\cos(x)-\cos(ux)).
\end{align*}
As $u$ is a positive integer, $g$ is a periodic function with period $2\pi.$ We have $g(0)=g(2\pi)=0.$ So, it is enough to show that $g(x)\geq 0$ for all $x\in [0,2\pi].$

Since $\cos(x)$ is decreasing function on $[0,\pi],$ we obtain that $\cos(x)> \cos (ux)$ for $x\in [0,\frac{\pi}{u}].$ Thus $g^{\prime \prime}(x)>0$ for $x\in [0,\frac{\pi}{u}].$ Consequently, $g^\prime (x)$ is strictly increasing on $[0,\frac{\pi}{u}].$ Also $g^{\prime}(0)=0.$ Thus we obtain that $g^{\prime}(x)>0$ for $x\in [0,\frac{\pi}{u}].$ As $\sin(x)$ is increasing on $[0,\frac{\pi}{2}],$ we get that $u\sin(x)-\sin(ux)\geq u\sin(\frac{\pi}{2u})-1= g^{\prime}(\frac{\pi}{2u})>0$  for all $x\in[\frac{\pi}{2u},\frac{\pi}{2}].$ Thus 
\begin{align}\label{gprime positive}
g^{\prime}(x)>0,\,\,\, x\in [0,\frac{\pi}{2}],
\end{align}
and consequently $g(x)$ is strictly increasing function on $[0,\frac{\pi}{2}].$ So, $g(x)>g(0)=0$ for $x\in(0,\frac{\pi}{2}].$

{\bf{Sub case 1:} $\,u=2m+1$} is  an odd integer,  $m\in\mathbb N:$
In this case it is straight forward to verify that $g^{\prime}(\frac{\pi}{2}-x)=g^{\prime}(\frac{\pi}{2}+x)$ for all $x.$ Since we already have $g^{\prime}(x)>0$ for $x\in [0,\frac{\pi}{2}],$ it follows that $g^{\prime}(x)>0$ for $x\in [0,\pi].$ Thus $g(x)$ is strictly increasing function on $[0,\pi].$ So, $g(x)>g(0)=0$ for $x\in(0,\pi].$

As $u$ is an odd integer, it follows that $g^{\prime}(\pi+x)=-g^{\prime}(x)$ for all $x.$ Since $g^{\prime}(x)>0$ for $x\in (0,\pi],$  we get that $g^{\prime}(x)<0$ for all $x\in (\pi,2\pi].$ Thus $g(x)$ is strictly decreasing function on $(\pi,2\pi].$ So, $g(x)>g(2\pi)=0$ for all $x\in[\pi,2\pi).$ Hence $g(x)\geq 0$ for all $x\in [0,2\pi].$

{\bf{Sub case 2:} $\,u=2m:$} is an even integer $2m$ for some $m\in\mathbb N:$ From \eqref{gprime positive} we have that  $g^{\prime}(x)>0$ for $x\in [0,\frac{\pi}{2}].$ Now we will show that $g^{\prime}(x)>0$ for $x\in [0,\pi).$ It is sufficient to show that $g^{\prime}(\pi-x) >0$ for $x\in (0,\frac{\pi}{2}].$

  As $u$ is an even integer, $g^{\prime}(\pi-x)= u(u\sin(x)+\sin(ux)).$
  So, $g^{\prime}(\pi-x) >0$ for $x\in (0,\frac{\pi}{2u}].$  Since $\sin(x)$ is increasing on $[0,\frac{\pi}{2}]$ and
$g^{\prime}(\frac{\pi}{2u})
=u(u\sin (\frac{\pi}{2u}) -1)>0$,  we obtain that
\begin{align*}
  u\sin(x)+\sin(ux)\geq u\sin(x)-1  >u\sin(\tfrac{\pi}{2u})-1=
 \tfrac{1}{u} g^{\prime}(\tfrac{\pi}{2u})>0,\,\,x\in[\tfrac{\pi}{2u},\tfrac{\pi}{2}].
\end{align*}
Thus $g^{\prime}(\pi-x) >0$ for $x\in (0,\frac{\pi}{2}].$ Hence $g^{\prime}(x)>0$ for $x\in [0,\pi).$ It follows that $g(x)$ is a increasing function on $[0,\pi)$ and $g(x)>g(0)=0$ for all $x\in(0,\pi].$

Now we will show $g^{\prime}(x)<0$ for $x\in(\pi,2\pi].$  Since $u$ is an even integer, it is straightforward to see that $u^{-1}g^{\prime}(\pi+x)=
-(u\sin(x)+\sin(ux))=-u^{-1}g^{\prime}(\pi-x)$ for all $x.$  As we already have $g^{\prime}(\pi-x)>0$ for $x\in (0,\frac{\pi}{2}],$ we get that $g^{\prime}(x) <0$ for $x\in(\pi,\frac{3\pi}{2}].$ Also note that
$u^{-1}g^{\prime}(2\pi-x)= -(u\sin(x)-\sin(ux))=-u^{-1}g^{\prime}(x).$ Using \eqref{gprime positive} we get that $g^{\prime}(x)<0$ for $x\in[\frac{3\pi}{2},2\pi).$ Hence $g^{\prime}(x)<0$ for $x\in(\pi,2\pi).$ So, $g(x)$ is strictly decreasing function on $(\pi,2\pi]$ and $g(x)>g(2\pi)=0$ for all $x\in[\pi,2\pi).$ Thus $g(x)\geq 0$ for all $x\in[0,2\pi].$

\end{proof}


\subsection*{Acknowledgement} The authors wishes to thank G. Misra for bringing up the attention to the special version of the question by N. Salinas and for his constant support in the preparation of this paper. The authors also express their sincere thanks to S. Chavan for his many useful comments and suggestions in the preparation of this article.

\bibliographystyle{amsplain} 

\begin{thebibliography}{10}

\bibitem{SameerMTP}
A.~Anand and S.~Chavan, \emph{Module tensor product of subnormal modules need
  not be subnormal}, J. Funct. Anal. \textbf{272} (2017), no.~11, 4752--4761.

\bibitem{SameerJI}
\bysame, \emph{A moment problem and joint {$q$}-isometry tuples}, Complex Anal.
  Oper. Theory \textbf{11} (2017), no.~4, 785--810.

\bibitem{Aronszajn1950theory}
N.~Aronszajn, \emph{Theory of reproducing kernels}, Transactions of the
  American mathematical society \textbf{68} (1950), no.~3, 337--404.

\bibitem{ball94cmrf}
  K.~Ball, \emph{Completely monotonic rational functions and {Hall}′ s marriage theorem},
  Journal of Combinatorial theory, Series B \textbf{61} (1994),
  no.~1, 118--124.

\bibitem{Beardon}
A.~F. Beardon, \emph{Sums of powers of integers}, Amer. Math. Monthly
  \textbf{103} (1996), no.~3, 201--213.

\bibitem{BergDuranSTHM}
  C.~Berg and A.~J. Dur{\'a}n, \emph{Some transformations of
    {Hausdorff} moment
  sequences and harmonic numbers}, Canadian Journal of Mathematics \textbf{57}
  (2005), no.~5, 941--960.

\bibitem{Conwaysubnormal}
J.~B. Conway, \emph{The theory of subnormal operators}, no.~36, American
  Mathematical Soc., 1991.

\bibitem{Douglasmodule}
R.~G. Douglas and V.~I. Paulsen, \emph{Hilbert modules over function algebras},
  Longman Sc \& Tech, 1989.

\bibitem{Hausdorff}
F.~Hausdorff, \emph{Momentprobleme f\"{u}r ein endliches {I}ntervall}, Math. Z.
  \textbf{16} (1923), no.~1, 220--248.

\bibitem{Paulsen2016rkhs}
V.~I. Paulsen and M.~Raghupathi, \emph{An introduction to the theory of
  reproducing kernel {Hilbert} spaces}, vol. 152, Cambridge University Press,
  2016.

\bibitem{Salinas1988products}
N.~Salinas, \emph{Products of kernel functions and module tensor products},
  Topics in operator theory, Oper. Theory Adv. Appl., vol.~32, Birkh\"{a}user,
  Basel, 1988, pp.~219--241.

\bibitem{Spitzbart}
A.~Spitzbart, \emph{A generalization of {H}ermite's interpolation formula},
  Amer. Math. Monthly \textbf{67} (1960), 42--46.

\bibitem{widderlaplace}
D.~V. Widder, \emph{Laplace transform (pms-6)}, Princeton university press,
  2015.

\end{thebibliography}
\providecommand{\bysame}{\leavevmode\hbox to3em{\hrulefill}\thinspace}
\providecommand{\MR}{\relax\ifhmode\unskip\space\fi MR }
\providecommand{\MRhref}[2]{%
  \href{http://www.ams.org/mathscinet-getitem?mr=#1}{#2}
}
\providecommand{\href}[2]{#2}

\end{document}